\documentclass[a4paper,12pt]{amsart}

\title[Symmetry properties]{Symmetry properties of some solutions to some semilinear elliptic equations}
\date{13/09/2014}

\usepackage{url}

\author{Alberto Farina}
\address[A.~Farina]{LAMFA, CNRS UMR 7352\\ Université de Picardie Jules Verne \\ 33 rue Saint-Leu\\ 80039 Amiens\\ France}
\email{alberto.farina@u-picardie.fr }
\author{Andrea Malchiodi}
\address[A.~Malchiodi]{Area of Mathematics\\ SISSA\\ Via Bonomea 265\\ 34136 Trieste\\ Italy}
\email{malchiod@sissa.it}
\author{Matteo Rizzi}
\address[M.~Rizzi]{Area of Mathematics\\ SISSA\\ Via Bonomea 265\\ 34136 Trieste\\ Italy}
\email{mrizzi@sissa.it}

\usepackage[T1]{fontenc}
\usepackage[latin1]{inputenc}
\usepackage[english]{babel}

\usepackage{indentfirst} 

\usepackage{amsmath,amsfonts,amssymb,amsthm}
\usepackage{latexsym}

\newcommand{\R}{\mathbb{R}}

\newcommand{\ol}{\overline{\lambda}}
\newcommand{\slk}{\Sigma_{\lambda_{k}}}
\newcommand{\uol}{u_{\overline{\lambda}}}
\newcommand{\sol}{\Sigma_{\overline{\lambda}}}
\newcommand{\ulk}{u_{\lambda_{k}}}
\newcommand{\ul}{\underline{\lambda}}

\newcommand{\stul}{\tilde{\Sigma_{\underline{\lambda}}}}
\newcommand{\stlk}{\tilde{\Sigma_{\lambda_{k}}}}
\newtheorem{theorem}{Theorem}
\newtheorem{proposition}[theorem]{Proposition}
\newtheorem{lemma}[theorem]{Lemma}
\newtheorem{corollary}[theorem]{Corollary}

\theoremstyle{remark}

\theoremstyle{remark}

\begin{document}

\maketitle

\begin{abstract}
In this paper we prove some symmetry results for entire solutions to the semilinear equation $-\Delta u=f(u)$, with $f$ nonincreasing in a right neighbourhood of the origin. We consider solutions decaying only in some directions and we give some sufficient conditions for them to be radially symmetric with respect to those variables, such as periodicity or the pointwise decay of some derivatives. 
\end{abstract}

\begin{center}

\bigskip\bigskip

\centerline{\bf AMS subject classification: 35J61,  35B07, 35B08, 35B09}

\end{center}

\section{Introduction}

In this paper we consider positive bounded solutions to the equation
\begin{eqnarray}
-\Delta u=f(u)
\label{eqp}
\end{eqnarray}
on $\mathbb{R}^{N}$. The nonlinearity will always be a $C^{1}$ function decreasing in a right neighborhood of the origin, that is
\begin{eqnarray}
f^{'}(s)\leq 0 &\text{ for $s\in(0,\varepsilon)$, for some $\varepsilon>0$, and $f(0)=0$.}
\label{condf}
\end{eqnarray}



The aim is to establish some symmetry results. In \cite{GNN} Gidas, Ni and Nirenberg proved the following theorem.
\begin{theorem}\cite{GNN}
Let $u>0$ be a solution to equation (\ref{eqp}), with $f$ satisfying condition (\ref{condf}). Assume furthermore that
\begin{eqnarray}
u(x)\to 0 &\text{as $|x|\to\infty$}.
\label{decunif}
\end{eqnarray}
Then, up to a translation, $u$ is radially symmetric and decreasing to $0$, that is $u=u(|x|)$,
with $\partial u/\partial r(x)<0$, for any $x\neq 0$.
\label{th1}
\end{theorem}
 
The main problem we are concerned with is the following: if we replace the decay hypothesis (\ref{decunif}) by the weaker assumptions that $u$ is bounded and satisfies 
\begin{eqnarray}
u(y,z)\to 0 &\text{ as $|z|\to\infty$, uniformly in $y$},
\label{decxn}
\end{eqnarray}

where we have set $x=(y,z)$, with $y\in\R^{M}$, $z\in\R^{N-M}$, is it true that $u$ is radially symmetric in $z$, that is $u=u(y,|z-z^{0}|)$, for some $z^{0}$, with $u_{j}(y,z)<0$ for $z_{j}>z^{0}_{j}$, $1\leq j\leq N-M$, where we have set $u_{j}=\partial u/\partial z_{j}$?

In the sequel, we will give some sufficient conditions for this to be true. An example of sufficient condition to get symmetry is periodicity in the $y$ variables.

We say that a function $u:\R^{N}\to\R$ is periodic in $y$ of period $T=(T_{1},\dots,T_{M})$ if, for any $(y,z)\in\R^{N}$, 
\begin{eqnarray}\notag
u(y+T_{j}e_{j},z)=u(y,z) &\text{for $1\leq j\leq M$}
\end{eqnarray}
where $\{e_{1},\dots,e_{M}\}$ denotes the standard basis in $\R^{M}$. 




\begin{theorem}
Let $u>0$ be a bounded solution to equation (\ref{eqp}), with $f$ satisfying (\ref{condf}). Let us write $x=(y,z)\in\R^{M}\times\R^{N-M}$, and assume that

$(i)$ $u$ is periodic in $y$

$(ii)$ $u(y,z)\to 0$ as $|z|\to\infty$, uniformly in $y$.

Then $u$ is radially symmetric and decreasing with respect to $z$, that is $u=u(y,|z-z_{0}|)$,
and $u_{j}(y,|z-z_{0}|)<0$ for $z^{j}>z^{j}_{0}$, $1\leq j\leq N-M$, for some $z_{0}\in\R^{N-M}$.
\label{th2}
\end{theorem}
\noindent\textbf{Remark 1.} In particular, in the case $M=0$, this result reduces to Theorem \ref{th1} by Gidas, Ni and Nirenberg, of which we give an alternative proof. 
\\

An interesting case is represented by the semilinear equation
\begin{eqnarray}
-\Delta u+u=u^{p}
\label{power}
\end{eqnarray}
with $1<p<\frac{N+1}{N-3}$ if $N>3$ and $p>1$ if $2\leq N\leq 3$. This equation arises naturally in several scientific contexts, for example as the nonlinear-Schrodinger equation in quantum mechanics but also biology, for instance in the study of the reaction-diffusion system proposed by Gierer and Meinhardt in 1972. For further informations, we refer to the papers \cite{GM,M}.

Dancer in \cite{D} showed that, for sufficiently large $T$, there exists a solution $u_{T}$ to (\ref{power}) fulfilling the following properties:
\begin{itemize}
\item $u_{T}(x)$ is even and periodic in $y$ with period $T$,
\item $u_{T}(x)$ is radially symmetric in $z$,
\item $u_{T}(y,z)\to 0$ exponentially fast as $|z|\to\infty$, uniformly in $y$, 
\end{itemize}
where we have set $x=(y,z)\in\R\times\R^{N-1}$.

Theorem \ref{th2}, in the case $M=1$, shows that any solution which is even and periodic in $y$ and decays in the other variables has to be symmetric in $z$, like Dancer' s solution. These results with periodicity will be proved in Sections $1$ and $2$.
\\

After that, we will consider solutions fulfilling (\ref{decxn}) and
\begin{eqnarray}
\text{for any $x_{N}$, }\nabla_{x^{'}}u(x^{'},x_{N})\to 0 &\text{as $|x|^{'}\to\infty$}.
\label{Neumann}
\end{eqnarray}
In order to investigate the behaviour of this kind of solutions, it is useful to study the problem
\begin{eqnarray}
\begin{cases}
-v^{''}=f(v) &\text{on $\R$}\\
v\geq 0\\
v(t)\to 0 &\text{as $|t|\to\infty$.}
\end{cases}
\label{ODE}
\end{eqnarray}
By the Cauchy uniqueness Theorem, $v>0$ or $v\equiv 0$. We will show that, if there exists a positive solution to (\ref{ODE}), then it is unique. It turns out that it is worth to distinguish the cases in which such a positive solution exists or not. 

\begin{theorem}
Let $f$ be a function fulfilling condition (\ref{condf}) such that problem (\ref{ODE}) admits no positive solution. Let $u>0$ be a bounded solution to
\begin{eqnarray}
\begin{cases}
-\Delta u=f(u) &\text{in $\R^{N}$}\\
u(x^{'},x_{N})\to 0 &\text{as $|x_{N}|\to\infty$, uniformly in $x^{'}$}\\
\nabla_{x^{'}}u(x^{'},x_{N})\to 0 &\text{as $|x^{'}|\to\infty$, for any $x_{N}$}
\end{cases}
\label{probn}
\end{eqnarray}
Then $u$ is radially symmetric, that is $u=u(|x-y|)$, for an appropriate $y\in\R^{N}$.
\label{thprof1d=0}
\end{theorem}

We observe that, if $f(t)=0$ for any $0<t<\delta$, for some $\delta>0$, then problem (\ref{ODE}) has no positive solution. In fact, for $t$ large enough, $v$ has to be affine, that is $v(t)=at+b$, but $v(t)\to 0$ as $t\to\infty$, hence $a=b=0$; by the Cauchy uniqueness theorem, $v\equiv 0$. As a consequence, in this case, Theorem \ref{thprof1d=0} holds true. For this kind of nonlinearities, in dimension $N=2$, we can get a non-existence result.
\begin{corollary}
Let $N=2$. Let $f$ be a $C^{1}(\R)$ function such that $f(t)=0$ for any $0<t<\delta$, for a suitable $\delta>0$. Then the only bounded solution $u\geq 0$ to (\ref{probn}) is $u\equiv 0$.
\label{nonex2}
\end{corollary}
A relevant example of nonlinearity of this type is $f(u)=((u-\beta)^{+})^{p}$, with $p>1$. In this case, when $N\geq 3$ and $1<p<\frac{N+2}{N-2}$, L. Dupaigne and A. Farina in \cite{DF} showed that the radially symmetric solution is unique and found the explicit expression 
\begin{eqnarray}\notag
u(x)=
\begin{cases}
\phi_{R}(|x|)+\beta &\text{for $|x|\leq R$}\\
\alpha|x|^{2-N} &\text{for $|x|\geq R$}
\end{cases}
\end{eqnarray}
where 
\begin{eqnarray}\notag
R=\bigg(\frac{1}{\beta(N-2)}\int_{0}^{1}\phi_{1}^{p}(r)r^{N-1}dr\bigg)^{(p-1)/2},
\end{eqnarray}
$\alpha=\beta R^{N-2}$ and $\phi_{R}$ is the unique radially symmetric and radially decreasing solution to the problem
\begin{eqnarray}\notag
\begin{cases}
-\Delta\phi_{R}=\phi_{R}^{p} &\text{for $|x|\leq R$}\\
\phi_{R}=0 &\text{on $|x|=R$}\\
\phi_{R}>0 &\text{in $|x|<R$}\\
\frac{\partial\phi_{R}}{\partial r}<0 &\text{for $0<|x|\leq R$}
\end{cases}
\end{eqnarray}
This example shows that, in dimension $N\geq 3$, Corollary \ref{nonex2} is not true.

With similar techniques, we obtain a lower bound for the $L^{\infty}$-norm of nontrivial solutions to equation (\ref{power}), decaying in one variable and fulfilling (\ref{Neumann}). In \cite{K}, Kwong showed that there exists a unique (up to a translation) positive radially symmetric solution to equation (\ref{power}), that we will denote by $U$. We observe that 
\begin{eqnarray}\notag
\max U>\bigg(\frac{p+1}{2}\bigg)^{\frac{1}{p-1}}.
\end{eqnarray}
In fact, 
up to a translation, we can assume that $\max U=U(0)$, that is $U(x)=v(|x|)$, where $v$ is a solution to the ODE
\begin{eqnarray}\notag
-v^{''}-\frac{N-1}{r}v^{'}(r)=f(v(r))
\end{eqnarray}
with $f(t)=t^{p}-t$. Multiplying the equation by $v^{'}$ and integrating, we get 
\begin{eqnarray}\notag
\frac{d}{dr}\bigg(\frac{1}{2}(v^{'}(r))^{2}+F(v(r))\bigg)=
-\frac{N-1}{r}(v^{'})^{2}<0
\end{eqnarray}
for $r>0$, with $F(t)=\frac{1}{p+1}t^{p+1}-\frac{1}{2}t^{2}$. So the energy $E(r)=\frac{1}{2}(v^{'}(r))^{2}+F(v(r))$ is strictly decreasing and $E(r)\to 0$ as $r\to\infty$. Therefore $E(r)>0$ for any $r$, in particular $E(0)=F(U(0))>0$, hence $\max U>(p+1/2)^{1/p-1}$.

This observation will be useful to prove the following proposition.

\begin{proposition}
Let $u>0$ be a bounded solution to equation (\ref{power}) satisfying condition (\ref{decxn}) with $z=x_{N}$. Assume that $\nabla_{x^{'}}u(x^{'},x_{N})\to 0$ as $|x^{'}|\to\infty$, for any $x_{N}$.
Then 
\begin{eqnarray}\notag
||u||_{\infty}\geq\bigg(\frac{p+1}{2}\bigg)^{1/p-1}.
\end{eqnarray}

\label{proprad}
\end{proposition}

Anyway, there are examples of nonlinearities for which problem (\ref{ODE}) admits a positive solution, such as $f(u)=|u|^{p-1}u-u$.
In order to deal with this case, we consider the energy-like functional
\begin{eqnarray}\notag
H(u,x^{'})=\int_{-\infty}^{\infty}\frac{1}{2}\big(u_{N}^{2}-
|\nabla_{x^{'}}u|^{2}\big)-F(u)dx_{N}
\end{eqnarray}  
and, for any $\lambda\in\mathbb{R}$, the momentum
\begin{eqnarray}\notag
E_{\lambda}(u,x^{'})=\int_{-\infty}^{\infty}(x_{N}-\lambda)\bigg(
\frac{1}{2}\big(u_{N}^{2}-|\nabla_{x^{'}}u|^{2}\big)-F(u)\bigg)dx_{N}.
\end{eqnarray} 
In the above definitions, we have denoted by $F(u)=\int_{0}^{u}f(t)dt$, the primitive of $f$ vanishing at the origin.
\\

\noindent\textbf{Remark 2.} If $f^{'}(0)<0$, condition (\ref{decxn}) with $z=x_{N}$ is sufficient for the energy and the momentum to be well defined and finite, since $u$ and $\nabla u$ actually decay exponentially in $x_{N}$, that is
\begin{eqnarray}
u(x),|\nabla u(x)|\leq Ce^{-\gamma|x_{N}|} &\text{for $|x_{N}|\geq M$,}
\label{expdecay}
\end{eqnarray}
for suitable constants $M>0$, $\gamma>0$.
\\

If $f^{'}(0)=0$, we need some further assumptions about $u$ in order for these defintions to be well posed, that is $|H(u,x^{'})|,|E_{\lambda}(u,x^{'})|<\infty$.
In this context, we require
\begin{eqnarray}
u(x)\leq C|x_{N}|^{-(1+\sigma)} &\text{for $|x_{N}|>M$}
\label{decay}
\end{eqnarray}
for suitable constants $M>0$, $\sigma>0$. We will show in section $3$ that this condition is sufficient for $H(u,x^{'})$ and $E_{\lambda}(u,x^{'})$ to be well defined and finite, provided $f\in C^{2}(\R^{N})$.

\begin{theorem}
Let $u>0$ be a bounded solution to equation (\ref{eqp}) satisfying condition (\ref{decay}), with $f\in C^{2}(\R)$ satisfying (\ref{condf}). Assume furthermore that

(a) There exists $\overline{x_{N}}\in\R$ and $\delta>0$ such that $u(x^{'},\overline{x_{N}})\geq\delta>0$, for any $x^{'}\in\R^{N-1}$.

(b) $\nabla_{x^{'}}u(x^{'},x_{N})\to 0$ as $|x^{'}|\to\infty$, for any $x_{N}$.


Then $u$ is symmetric in $x_{N}$, that is $u=u(x^{'},|x_{N}-\lambda|)$, for some $\lambda\in\R$, and $u_{N}(x^{'},x_{N})>0$ if $x_{N}<\lambda$.

\label{thndim}
\end{theorem}

\noindent\textbf{Remark 3.} In Theorem \ref{thndim}, we can assume that there exists a positive solution to Problem (\ref{ODE}), otherwise, by Theorem \ref{thprof1d=0}, there are no solutions $u$ fulfilling hypothesis of Theorem \ref{thndim}. 
\\

Section $3$ will be devoted to the proof of this theorem, that holds true in any dimension $N\geq 2$. In Theorem \ref{thndim}, we would like to be able to drop assumption (a). Up to now, we have been able to do so only in dimension $N=2$.

\begin{theorem}
Let $N=2$. Let $u>0$ be a bounded solution to equation (\ref{eqp}) satisfying condition (\ref{decay}), with $f\in C^{2}(\R)$ satisfying (\ref{condf}). Assume furthermore that $u_{1}(x_{1},x_{2})\to 0$ as $|x_{1}|\to\infty$, for any $x_{2}$.


Then $u$ is symmetric in $x_{2}$, that is $u=u(x_{1},|x_{2}-\lambda|)$, for some $\lambda\in\R$.
\label{th2dim}
\end{theorem}

\noindent\textbf{Remark 4.} If $f^{'}(0)<0$, thanks to (\ref{expdecay}), Theorems \ref{thndim} and \ref{th2dim} hold true even if we replace condition (\ref{decay}) with the weaker assumption (\ref{decxn}). 
\\

\noindent\textbf{Remark 5.} In dimension $N=2$, Theorem \ref{th2dim} is a extension to Theorem $1.1$ of \cite{BF} to more general nonlinearities, since we do not need to take $f(u)=u+g(u)$, with $g$ satisfying their assumptions $(f1)$, $(f2)$ and $(f3)$. On the other hand, we need some more regularity, we take $f\in C^{2}$ instead of $C^{1,\beta}$. 
\\

Unfortunately, if $f$ is flat near the origin, condition (\ref{decxn}) does not necessarily imply (\ref{decay}), at least in dimension $N\geq 3$. In fact, the solution constructed by L. Dupaigne and A. Farina in \cite{DF} in dimension $N=3$ decays as $|x|^{-1}$ (see the above discussion for the explicit expression). This function, seen as a solution in higher dimension, is a counter-example in dimension $N\geq 4$ too.

In Section $5$, we consider solutions to (\ref{eqp}) decaying in $N-1$ variables, and we prove the following theorem.

\begin{theorem}
Let $N\geq 5$. Let $u>0$ be a bounded solution to equation (\ref{eqp}), 
with $f\in C^{2}(\R)$ satisfying (\ref{condf}). Assume that
\begin{eqnarray}
u(x^{'},x_{N})\to 0 &\text{as $|x^{'}|\to\infty$, uniformly in $x_{N}$}
\label{limx'}
\end{eqnarray}
and 
\begin{eqnarray}
&\text{for some $x^{'}_{0}$, }u(x^{'}_{0},x_{N})\to 0 &\text{as $x_{N}\to\infty$ }.
\label{decn}
\end{eqnarray}
Then $u$ is radially symmetric. 
\label{thdec}
\end{theorem}

\noindent\textbf{Remark 6.} We observe that, if we assume $f\in C^{1}$ with $f^{'}(0)<0$, then, thanks to the exponential decay (apply (\ref{expdecay}) $N-1$ times), Theorem \ref{thdec} holds true in any dimension $N\geq 2$. 
\\

\noindent\textbf{Remark 7.} In dimension $2\leq N\leq 4$, Theorem \ref{thdec} holds true under the assumption 
\begin{eqnarray}
u(x),|\nabla u(x)|\leq C|x^{'}|^{-\frac{N-1+\sigma}{2}} &\text{for $|x^{'}|\geq M$}
\label{decN}
\end{eqnarray}
for suitable constants $M>0$, $\sigma>0$ and $f\in C^{1}$.
\\ 

In order to deal with the case $f^{'}(0)=0$, we study the decay rate at infinity of functions fulfilling (\ref{limx'}). This will be carried out in section $6$.


\

\noindent {\bf Acknowledgements} A.F. is partially supported by the ERC grant EPSILON (Elliptic Pde's and
Symmetry of Interfaces and Layers for Odd Nonlinearities) and by the
ERC grant COMPAT (Complex Patterns for Strongly Interacting Dynamical
Systems).  A.M. and M.R. have been supported by the PRIN project {\em Variational and perturbative aspects of nonlinear differential problems}.

\section{Starting the moving plane procedure}

First we define, for $\lambda\in\mathbb{R}$, $u_{\lambda}(x)=u(x^{'},2\lambda-x_{N})$, $\Sigma_{\lambda}=\{x_{N}<\lambda\}$. 
In the following proposition, we prove that the moving plane procedure can be started. In order to do so, it is enough to replace condition (\ref{decxn}) with the weaker assumption
\begin{eqnarray}
u(x^{'},x_{N})\leq\varepsilon &\text{in the subspace $\{x_{N}>\lambda_{0}\}$}
\label{smallxN}
\end{eqnarray} 
for a suitable $\lambda_{0}\in\R$, if $f$ is nonincreasing in the interval $(0,\varepsilon)$. 
\begin{proposition}
Let $u>0$ be a bounded solution to equation (\ref{eqp}) fulfilling (\ref{smallxN}). Assume that $f$ satisfies (\ref{condf}). Then $u-u_{\lambda}\geq 0$ in $\Sigma_{\lambda}$, for any $\lambda\geq\lambda_{0}$.
\label{propstart}
\end{proposition}
\noindent\textbf{Remark 8.} In particular, this proposition holds true if we assume that 
\begin{eqnarray}\notag
u(x^{'},x_{N})\to 0 &\text{as $x_{N}\to\infty$ uniformly in $x^{'}$.}
\end{eqnarray}

\begin{proof}
We assume by contradiction that it is possible to find $\lambda\geq\lambda_{0}$ such that the open set $\Omega_{\lambda}=\{u-u_{\lambda}<0\}\cap\Sigma_{\lambda}$ is not empty. By the monotonicity of $f$ near the origin, we get that, for any nonempty connected component $\omega$ of $\Omega_{\lambda}$,
\begin{eqnarray}\notag
\begin{cases}
-\Delta(u-u_{\lambda})=f(u)-f(u_{\lambda})\geq 0 & \text{in $\omega$}\\
u-u_{\lambda}=0 & \text{on $\partial\omega$}.
\end{cases} 
\end{eqnarray}
Hence, by the maximum principle for possibly unbounded domains (see \cite{BCN}, Lemma $2,1$), we conclude that $u-u_{\lambda}\geq 0$ in $\omega$, a contradiction.

\end{proof}


In view of this proposition, we can define

\begin{eqnarray}
\ol=\inf\{\lambda_{0}: u-u_{\lambda}\geq 0\text{ in $\Sigma_{\lambda}$,} \forall\lambda\geq\lambda_{0}\}.
\label{defol}
\end{eqnarray}
By construction, we see that $\ol<\infty$.

\begin{lemma}
Let $u\geq 0$ be a bounded solution to equation (\ref{eqp}) fulfilling (\ref{smallxN}). Assume that $f$ satisfies (\ref{condf}). 

$(i)$ If $\ol=-\infty$, then $u_{N}\equiv 0$ or $u_{N}(x)<0$ for any $x\in\R^{N}$.

$(ii)$ If $u$ satisfies condition (\ref{decxn}) and $\ol=-\infty$, then $u\equiv 0$.

$(iii)$ If $u$ satisfies condition (\ref{decxn}) and $f^{'}(t)\leq 0$ for any $t>0$, then $u\equiv 0$.
\label{lemmaol}
\end{lemma} 
\begin{proof}
$(i)$ If $\ol=-\infty$, that is the moving plane method does not stop, then $u_{N}\leq 0$. Since $u_{N}$ verifies the linearized equation $-\Delta u_{N}=f^{'}(u)u_{N}$, by the strong maximum principle, we get that $u_{N}\equiv 0$ or $u_{N}<0$ in the whole $\R^{N}$.

$(ii)$ If $\ol=-\infty$, the monotonicity, together with condition (\ref{decxn}), yields that $u\equiv 0$.

$(iii)$ If $f^{'}(t)\leq 0$ for any $t>0$, then $\ol=-\infty$, hence, by statement $(ii)$, $u\equiv 0$.

\end{proof}


\begin{proposition}
Let $u>0$ be a bounded solution to equation (\ref{eqp}) fulfilling (\ref{smallxN}). Assume that $f$ satisfies (\ref{condf}). Assume, in addition, that $\ol>-\infty$.

$(i)$ For any positive integer $k$, there exists $\ol-1/k\leq\lambda_{k}<\ol$ and a point $x^{k}\in\slk$, with $\{x^{k}_{N}\}$ bounded, such that
\begin{eqnarray}
 u(x^{k})<u_{\lambda_{k}}(x^{k})
\label{relol}
\end{eqnarray}

$(ii)$  If, in addition, $u$ is periodic in $x_{N}$, then the sequence $x^{k}$ can be chosen to be bounded.
\label{profile}
\end{proposition}
\begin{proof}
$(i)$ It follows from the definition of $\ol$ that we can choose a sequence $\ol-1/k\leq\lambda_{k}<\ol$ and a point $x^{k}\in\slk$ such that $u(x^{k})<u_{\lambda_{k}}(x^{k})$. By construction, we have that $x^{k}_{N}<\lambda_{k}<\ol$; what remains to prove is that we can choose these sequences in such a way that $x^{k}_{N}$ is bounded from below.
 We define
\begin{eqnarray}\notag
\Lambda=\{\big((\lambda_{k})_{k},(x^{k})_{k}\big):\ol-1/k\leq\lambda_{k}<\ol
\text{, }x^{k}\in\slk\text{ and }u(x^{k})<u_{\lambda_{k}}(x^{k})\}
\end{eqnarray}
and we argue by contradiction. We assume that for any couple of sequences $(\tilde{\lambda},\tilde{x})=\big((\lambda_{k})_{k},(x^{k})_{k}\big)\in\Lambda$, we have $x^{k}_{N}\to-\infty$. Hence, once we fix $M>0$ and such a couple $(\tilde{\lambda},\tilde{x})$, we can find $\overline{k}$ such that $x^{k}_{N}<-M$, for $k\geq\overline{k}$. Now, if we set
\begin{equation}\notag
k_{0}(\tilde{\lambda},\tilde{x})=\min\{\overline{k}:x^{k}_{N}<-M\text{, for } k\geq\overline{k}\},
\end{equation} 
we have that $x^{k}_{N}<-M$ for $k\geq k_{0}(\tilde{\lambda},\tilde{x})$, while $x^{k_{0}(\tilde{\lambda},\tilde{x})-1}\geq-M$.

After that we set
\begin{eqnarray}\notag
k_{0}=\sup\{k_{0}(\tilde{\lambda},\tilde{x}):(\tilde{\lambda},\tilde{x})\in
\Lambda\};
\end{eqnarray}
if $k_{0}=\infty$, the family $\{k_{0}(\tilde{\lambda},\tilde{x}):(\tilde{\lambda},\tilde{x})\in
\Lambda\}$ would be a diverging sequence $k_{j}$ of positive integers, that we can assume to be increasing and such that $k_{j}>j$. For any $j$, we set $i=k_{j}-1$ and consider the corresponding couple $(\tilde{\lambda},\tilde{x})$: we set $\mu_{i}=\lambda_{i}$ and $s^{i}=x^{i}$. The couple $(\mu,s)$ still belongs to $\Lambda$ and $s^{i}_{N}\geq-M$, a contradiction.

Therefore, we have that $k_{0}<\infty$ and, for any $k\geq k_{0},$ $u-\ulk\geq 0$ in $\{-M<x_{N}<\lambda_{k}\}$. Now, if we choose $M$ so large that $u(x)<\varepsilon$ for $x_{N}>2(\ol-1)-M$, we have, for $k\geq k_{0}$
\begin{eqnarray}\notag
\begin{cases}
-\Delta(u-u_{\lambda_{k}})=f(u)-f(u_{\lambda_{k}})\geq 0 & \text{in $\omega$}\\
u-u_{\lambda_{k}}=0 & \text{on $\partial\omega$},
\end{cases} 
\end{eqnarray}
where $\omega$ is any connected component of the set $\Omega_{k}=\{x_{N}<-M\}\cap\{u-u_{\lambda_{k}}<0\}$. Therefore, by the maximum principle for possibly unbounded domains (see \cite{BCN}, Lemma $2,1$), we get that $u-\ulk\geq 0$ in $\omega$, hence $\Omega_{k}=\emptyset$, that is $u-\ulk\geq 0$ in $\slk$, for $k\geq k_{0}$.

The same is true for any $\lambda>\lambda_{k_{0}+1}$. Otherwise, we would be able to find a couple $(\lambda,x^{\lambda})$ such that $u(x^{\lambda})<u_{\lambda}(x^{\lambda})$, with $x^{\lambda}\in\Sigma_{\lambda}$ and $\lambda>\lambda_{k_{0}+1}$. As a consequence, $\lambda=\tilde{\lambda}_{k_{0}+1}$, for an appropriate $\tilde{\lambda}$, so $u-u_{\lambda}\geq 0$ in $\Sigma_{\lambda}$, which is not possible.

$(ii)$ It follows from the periodicity that we can redefine $x^{k}$ in order for $(x^{'})^{k}$ to be bounded.
\end{proof}

\section{Results with periodicity}

Now we can proceed with the proof of Theorem \ref{th2} in the case $M=N-1$.

\begin{proof}
As first we note that, by statement $(ii)$ of Lemma \ref{lemmaol}, $\ol>-\infty$, otherwise $u\equiv 0$. Since $u-\uol\geq 0$ in $\Sigma_{\ol}$, the strong maximum principle yields that either $u\equiv \uol$ or $u>\uol$ in $\Sigma_{\ol}$. Now we argue by contradiction and assume that the second possibility holds true. We take a sequence of real numbers $\lambda_{k}$ and a sequence of points $x^{k}\in\Sigma_{\lambda_{k}}$ as in Proposition \ref{profile}. By the boundedness of $x^{k}$, we have that, up to a subsequence, $x^{k}\to x^{\infty}$, so, by the (\ref{relol}), we get that $u(x^{\infty})\leq u_{\ol}(x^{\infty})$. Since we are assuming that $u>u_{\ol}$ in $\Sigma_{\ol}$, we have that $x^{\infty}_{N}=\ol$. By the Hopf Lemma, we obtain that $u_{N}(x^{'},\ol)<0$, but the mean value theorem yields that
\begin{eqnarray}\notag
0<\ulk(x^{k})-u(x^{k})=2(\lambda_{k}-x^{k}_{N})u_{N}((x^{'})^{k},\xi^{k})
\end{eqnarray}
with $x^{k}_{N}<\xi^{k}<2\lambda^{k}-x^{k}_{N}$.
Letting $k\to\infty$, we conclude that $u_{N}(x^{\infty})\geq 0$, a contradiction. Hence we have $u=\uol$ in $\Sigma_{\ol}$.
\end{proof}

Now let us consider the general case. In next proposition, hypothesis $(ii)$ of Theorem \ref{th2} can be replaced by the weaker assumptions
\begin{eqnarray}
\begin{cases}
u(y,z^{'},z_{N})\to 0 &\text{as $|z^{'}|\to\infty$, uniformly in the other variables}\\
u(y,z^{'},z_{N})\to 0 &\text{as $z_{N}\to\infty$, uniformly in the other variables.}
\end{cases}
\label{limrelaxed}
\end{eqnarray}
Under these hypotheses,
it is possible to define $\ol<\infty$ as before. 
\begin{proposition}
Let $u>0$ be a bounded solution to equation (\ref{eqp}) satisfying condition (\ref{limrelaxed}). Assume that $f$ satisfies (\ref{condf}). Assume furthermore that $\ol>-\infty$.

$(i)$ Then, for any positive integer $k$, there exists $\ol-1/k\leq\lambda_{k}<\ol$ and a point $x^{k}=(y^{k},z^{k})\in\slk$, with $\{z^{k}\}$ bounded, such that
\begin{eqnarray}\notag
 u(x^{k})<u_{\lambda_{k}}(x^{k})
\end{eqnarray}

$(ii)$ If, in addition, $u$ is periodic in $y$, then the sequence $x^{k}$ can be taken in such a way that it is bounded.
\label{profile1}
\end{proposition}
This is a generalisation of Proposition \ref{profile}, for which we have nevertheless presented an independent proof. 
\begin{proof}
As in the proof of Proposition \ref{profile}, by definition of $\ol$, we can find a sequence of real numbers $\ol-1/k\leq\lambda_{k}<\ol$ and a sequence of points $x^{k}=(y^{k},z^{k})\in\slk$ such that (\ref{relol}) holds. The difference is that now we want to prove that this sequence can be chosen in such a way that $z^{k}$ is bounded. In order to do so we will argue by contradiction. By construction, we know that $z^{k}_{N}\leq\ol$. In the notation of Proposition \ref{profile}, we define, for $R>0$ and $(\tilde{\lambda},\tilde{x})\in\Lambda$, the number
\begin{eqnarray}\notag
k_{0}(R,\tilde{\lambda},\tilde{x})=\inf\{k_{0}:z^{k}_{N}\leq-R,|(z^{'})^{k}|\geq R\text{ }\forall k\geq k_{0}\}.
\end{eqnarray}
Now we put
\begin{eqnarray}\notag
k_{0}(R)=\sup\{k_{0}(R,\tilde{\lambda},\tilde{x})\};
\end{eqnarray}
exactly as in Proposition \ref{profile}, we get that $k_{0}(R)<\infty$, for any $R>0$ and $u-\ulk\geq 0$ in $\Sigma_{\lambda_{k}}\cap Q_{R}$ for any $k\geq k_{0}$, where we have set $Q_{R}=\{|z^{'}|\leq R,\text{ }z_{N}\geq-R \}$.

By the decay assumptions, if $R$ is large enough, we have that $u(y,z)<\varepsilon$ for $|z^{'}|>R$ and $\ulk(y,z)<\varepsilon$ for $z_{N}<-R$ and for any $k$. Hence, if
we set $\Omega_{k}=\{u-\ulk<0\}\cap\Sigma_{\lambda_{k}}$, we get that, for any connected component $\omega$ of $\Omega_{k}$,
\begin{eqnarray}\notag
\begin{cases}
-\Delta(u-u_{\lambda_{k}})=f(u)-f(u_{\lambda_{k}})\geq 0 & \text{in $\omega$}\\
u-u_{\lambda_{k}}=0 & \text{on $\partial\omega$},
\end{cases} 
\end{eqnarray}
hence, by rhe maximum principle for possibly unbounded domains (see \cite{BCN}, Lemma $2,1$), $\omega=\emptyset$, as desired.

\end{proof}

The conclusion of the proof of Theorem \ref{th2} is similar to what we have done in the case $M=N-1$. As first we observe that, by the behaviour of $u$ for $z_{N}\to-\infty$, applying Lemma \ref{lemmaol}, we get $\ol>-\infty$. Then we take a sequence $x^{k}=(y^{k},z^{k})$ as in Proposition \ref{profile1}; up to a subsequence, we can assume that $x^{k}\to x^{\infty}=(y^{\infty},z^{\infty})$. Passing to the limit in (\ref{relol}), we can see that $u(y^{\infty},z^{\infty})\leq\uol(y^{\infty},z^{\infty})$. If $u>\uol$ in $\Sigma_{\ol}$, we get that $(y^{\infty},z^{\infty})\in\partial\Sigma_{\ol}$, but this contradicts the Hopf Lemma, as we have seen above.


\section{Results without periodicity}

As first we observe that condition (\ref{Neumann}) enables us to relate the study of equation (\ref{eqp}) to the study of one dimensional problem (\ref{ODE}).
The results concerning this one-dimensional problem are probably known, for sake of completeness we report the proofs. 

Before giving these proofs, let us fix some terminology. If $u$ is a bounded solution to (\ref{eqp}), then for any sequence $|x^{k}|\to\infty$, it is possible to find a subsequence such that $u^{k}(x)=u(x+x^{k})\to u^{\infty}(x)$ in the $C^{2}_{loc}(\R^{N})$ sense, and $u^{\infty}$ is still a solution. In the sequel, this kind of solutions, obtained as a limit of sequences constructed as above, will be referred to as \emph{profiles}. In the sequel, we will say that a profile is one dimensional if it is a function depending just on the $x_{N}-$variable.
\begin{lemma}
Let $u$ be a bounded solution to equation (\ref{eqp}) satisfying (\ref{decxn}) with $z=x_{N}$, and with $f$ fulfilling (\ref{condf}). Then any profile is one dimensional if and only if (\ref{Neumann}) holds.
\label{lemmaprof1d}
\end{lemma}
\begin{proof}
If any profile is one dimensional, for any $|(x^{'})^{k}|\to\infty$, there is a subsequence such that $u^{k}(x)=u(x^{'}+(x^{'})^{k},x_{N})\to v(x)$ in $C^{2,\alpha}_{loc}(\R^{N})$, with $v_{j}\equiv 0$, for $1\leq j\leq N-1$. This implies, in particular, that $u_{j}^{k}\to 0$ pointwise, therefore $u_{j}((x^{'})^{k},x_{N})\to 0$ for any $x_{N}\in\R$. Since the sequence $(x^{'})^{k}$ is arbitrary, we conclude that $u_{j}(x^{'},x_{N})\to 0$ as $|x^{'}|\to\infty$, for any $x_{N}$. 

The converse is true because $C^{2}_{loc}$ convergence implies pointwise convergence.
\end{proof}
Now we are going to study Problem (\ref{ODE}). For solutions satisfying
\begin{eqnarray}
v(t)\leq C|t|^{-(1+\sigma)} &\text{for any $|t|\geq M$}
\end{eqnarray}
for suitable constants $M>0$, $\gamma>0$, we define 
\begin{eqnarray}\notag
H(v)=\int_{-\infty}^{\infty}\frac{1}{2}(v^{'})^{2}-F(v)dt
\end{eqnarray}
and, for any $\lambda\in\R$,
\begin{eqnarray}\notag
E_{\lambda}(v)=\int_{-\infty}^{\infty}(t-\lambda)\Big(\frac{1}{2}
(v^{'})^{2}-F(v)\Big)dt.
\end{eqnarray}
In order to show that $H(v)$ and $E(v)$ are well defined and finite for such solutions, we prove the following lemma.

\begin{lemma}
Let $v>0$ be a solution to Problem (\ref{ODE}). Then

$(i)$ $v$ is symmetric with respect to $\lambda$, for some $\lambda\in\R$, and $v^{'}(t)>0$ for any $t<\lambda$.

$(ii)$ For any $t\in\R$, we have $\frac{1}{2}(v^{'}(t))^{2}+F(v(t))=0$.

$(iii)$ If we assume, in addition, that $v(t)\leq C|t|^{-(1+\sigma)}$, for some $\sigma>0$, then $|v^{'}(t)|\leq C|t|^{-(1+\sigma)}$, for any $|t|\geq M$.
\label{lemmaODE}
\end{lemma}
\begin{proof}
$(i)$ Since $v(t)\to 0$ as $t\to\infty$, the solution must have a maximum point at $t=\lambda$, for some $\lambda\in\R$. In particular it satisfies the Cauchy problem
\begin{eqnarray}\notag
\begin{cases}
-v^{''}=f(v) &\text{on $\R$}\\\notag
v(\lambda)=v_{max}\\\notag
v^{'}(\lambda)=0.
\end{cases}
\end{eqnarray}
A computation shows that $v_{\lambda}(t)=v(2\lambda-t)$ satisfies the same Cauchy problem, hence $v_{\lambda}=v$. If $v$ had another critical point $\mu\neq\lambda$, it would also be symmetric with respect to $\mu$, and hence periodic, but this is not possible because it tends to $0$ at infinity.

$(ii)$ Multiplying the ODE by $v^{'}$ and integrating we obtain the relation
\begin{eqnarray}\notag
0\leq\frac{1}{2}(v^{'})^{2}=-F(v)+C,
\end{eqnarray}
where $C$ is a suitable constant. Letting $t\to\infty$, we get that $C\geq 0$. If we had $C>0$, we would get that $(v^{'})^{2}\to 2C>0$ as $t\to\infty$, which is not possible because $v\to 0$ as $t\to\infty$. Finally we get that $C=0$ and $v^{'}\to 0$ as $t\to\infty$.

$(iii)$ If $f^{'}(0)<0$, the claim follows from the exponential decay of the derivative, so we can assume that $f^{'}(0)=0$.
We assume by contradiction that for any positive integer $k$, we can find $|t_{k}|>k$ such that $|v^{'}(t_{k})|>k|t_{k}|^{-(1+\sigma)}$. Now we set
\begin{eqnarray}\notag
v^{k}(t)=|t_{k}|^{\sigma}v(|t_{k}|t)
\end{eqnarray}
A computation shows that, for $k$ large enough and for any $\frac{1}{2}<|t|<2$, we have
\begin{eqnarray}\notag
(v^{k})^{'}(t)=|t_{k}|^{1+\sigma}\sqrt{-2F(v(|t_{k}|t))}\leq|t_{k}|^{1+\sigma}
\sqrt{2C|t_{k}|^{-2(1+\sigma)}|t|^{-2(1+\sigma)}}\leq C.
\end{eqnarray}
However, we can see that
\begin{eqnarray}
(v^{k})^{'}(t_{k}/|t_{k}|)\geq|t_{k}|^{1+\sigma}k|t_{k}|^{-(1+\sigma)}=k,
\end{eqnarray}
a contradiction.
\end{proof}

\begin{proposition}
If there exists a nontrivial solution to Problem (\ref{ODE}), then it is unique up to a translation.
\label{propODE}
\end{proposition}
It follows from the Cauchy uniqueness theorem that any nontrivial solution to Problem (\ref{ODE}) is strictly positive. Nevertheless, we point out that a nontrivial solution does not always exist, for instance if $f(u)=((u-\beta)^{+})^{p}$ with $\beta>0$, as we will see later.
\begin{proof}


Let us assume that there are two solutions $v>0$ and $w>0$, that are not one the translated of the other. Up to a translation, we can assume that the symmetry axes are the same, that is there exists $\lambda\in\R$ such that $v=v(|t-\lambda|)$ and $w=w(|t-\lambda|)$. 

If $v(\lambda)=w(\lambda)$, then we also have $v^{'}(\lambda)=w^{'}(\lambda)=0$, since $\lambda$ is a maximum point for both $v$ and $w$; therefore, by the Cauchy uniqueness theorem, we get that $v\equiv w$.

Now, assume, for instance, that $w(\lambda)>\ v(\lambda)$. By continuity, there exists $t_{0}>\lambda$ such that $w(t_{0})=v(\lambda)$. As a consequence, we conclude that
\begin{eqnarray}\notag
0>w^{'}(t_{0})=\sqrt{-2F(w(t_{0}))}=\sqrt{-2F(v(\lambda))}=v^{'}(\lambda),
\end{eqnarray}  
a contradiction.
\end{proof}

Now, let us prove a quite general Lemma, in which we do not need to assume that $u$ is a solution to some PDE.
 
\begin{lemma}
Let us denote $x=(y,z)\in\R^{M}\times\R^{N-M}$. Let $u:\R^{N}\to\R$ be a continuous function such that 
\begin{eqnarray}\notag
u(y,z)\to 0 &\text{as $|z|\to\infty$, uniformly in $y$.}
\end{eqnarray}

$(i)$ Assume that for any sequence $|y^{k}|\to\infty$ it is possible to find a subsequence such that $u^{k}(x)=u(y+y^{k},z)\to 0$ in the $C^{0}_{loc}$ sense. Then $u(x)\to 0$ as $|x|\to\infty$.

$(ii)$ Let $M=1$. Assume that for any sequence $y^{k}\to\infty$ it is possible to find a subsequence such that $u^{k}(x)=u(y+y^{k},z)\to 0$ in the $C^{0}_{loc}$ sense. Then
\begin{eqnarray}
u(y,z)\to 0 &\text{as $y\to\infty$, uniformly in $z$}.
\end{eqnarray}
\label{lemmasolrad}
\end{lemma}
\begin{proof}
$(i)$ By the decay in $z$, we have that, for any $\varepsilon>0$, there exists $M>0$ such that $u(y,z)<\varepsilon$ for $|z|\geq M$. Since $u^{k}\to u^{\infty}$ in the $C^{0}_{loc}$ sense, the convergence is uniform in the compact set $K=\{|z|\leq M,y=0\}$. Hence for any sequence $|y^{k}|\to\infty$, there is a subsequence such that
\begin{eqnarray}\notag
\sup_{K}|u^{k}(x)|=\sup_{|z|\leq M}|u(y^{k},z)|\to 0,
\end{eqnarray}
therefore $u(y,z)\to 0$ as $|y|\to\infty$, uniformly in $z$, so we have the statement.

$(ii)$ We essentially repeat the same proof, with the only difference that we consider only sequences $y^{k}\to\infty$.
\end{proof}

Now we are going to prove Theorem \ref{thprof1d=0}.

\begin{proof}


For any sequence $|(x^{'})^{k}|\to\infty$, by Lemma \ref{lemmaprof1d}, any corresponding profile $v$ is one-dimensional and satisfies (\ref{ODE}), so, by our assumption about $f$, $v\equiv 0$. Since this is true for any profile, Lemma \ref{lemmasolrad} yields that $u(x)\to 0$ as $|x|\to\infty$, hence, by the result by Gidas, Ni, Nirenberg in \cite{GNN}, $u$ is radially symmetric and radially decreasing. 




\end{proof}
Now we can prove Corollary \ref{nonex2}
\begin{proof}
Assume by contradiction that such a solution exists. Then by Theorem \ref{thprof1d=0} it is radially symmetric, that is, up to a translation, $u(x)=v(|x|)$, and harmonic outside a ball, so $u(x)=a\log(|x|)+b$, for $|x|$ large enough. If $a=0$, by (\ref{decxn}), we get $b=0$, so $u\equiv 0$. Otherwise, $a\neq 0$ and $b\in\R$, but this contradicts condition (\ref{decxn}).
\end{proof}
Now we can prove Proposition \ref{proprad}.
\begin{proof}
In the proof, we set $F(u)=\frac{1}{p+1}|u|^{p+1}-\frac{1}{2}u^{2}$ and $f(u)=|u|^{p-1}u-u$.

In order to prove the proposition, we will assume by contradiction that $0<||u||_{\infty}<(p+1/2)^{1/p-1}$ and we will see that this yields that for any $|(x^{'})^{k}|\to\infty$, the corresponding profile is identically $0$, hence, by Lemma \ref{lemmasolrad}, up to a translation, $u(x)=U(|x|)$, but, by our assumption, we have that $||u||_{\infty}<(p+1/2)^{1/p-1}<\max U$, a contradiction.

As before, by Lemma \ref{lemmaprof1d}, we get that any profile is one dimensional and satisfies (\ref{ODE}), therefore, by point $(i)$ of Lemma \ref{lemmaODE}, we know that $v=v(|t-\lambda|)$, for an appropriate $\lambda\in\R$. By symmetry, we get that $v^{'}(\lambda)=0$, hence, by point $(ii)$ of Lemma \ref{lemmaODE}, $F(v(\lambda))=0$. Anyway, we have that $v(\lambda)=||v||_{\infty}\leq||u||_{\infty}<(p+1/2)^{1/p-1}$, so we conclude that $||v||_{\infty}=0$.

\end{proof}

\section{Proofs of theorems $7$ and $8$}

In this section we are going to deal with the cases in which problem (\ref{ODE}) has a positive solution, so we can have a positive profile when we translate in the $x^{'}$-directions. In order to deal with this case, we need to consider the energy $H(u,x^{'})$ and the momentum $E_{\lambda}(u,x^{'})$ of a solution, hence we need some further assumptions about the decay rate of $u$ in $x_{N}$. In next lemma, we see that it is enough to prescribe the decay rate of $u$, we do not need any further assumption about the gradient. 

\begin{lemma}

Let $u>0$ be a bounded solution to (\ref{eqp}) with $f\in C^{2}(\R)$ satisfying (\ref{condf}) and $f^{'}(0)=0$. Assume furthermore that $u(x)\leq C|x_{N}|^{-\alpha}$ for $|x_{N}|\geq M$, for some constants $M>0$ and $\alpha\geq 1$. Then 

$(i)$ the gradient satisfies 
\begin{eqnarray}
|\nabla u(x)|\leq C|x_{N}|^{-\alpha} &\text{for $|x_{N}|\geq M$.}
\label{gradestxN}
\end{eqnarray}

$(ii)$ If $\alpha\geq 2$, then
\begin{eqnarray}
|\nabla u(x)|\leq C|x_{N}|^{-(1+\alpha)} &\text{for $|x_{N}|\geq M$.}
\label{gradestxN+1}
\end{eqnarray}

\label{lemmaestgrad}
\end{lemma}
\begin{proof}
$(i)$ Assume by contradiction that (\ref{gradestxN}) fails. Then it is possible to find a sequence of points $x^{k}\in\R^{N}$, with $|x_{N}^{k}|\geq k$, such that
\begin{eqnarray}\notag
|\nabla u((x^{'})^{k},x_{N}^{k})|\geq k|x_{N}^{k}|^{-\alpha}.
\end{eqnarray}
Now we define

\begin{eqnarray}\notag
v^{k}(x^{'},x_{N})=|x_{N}^{k}|^{\alpha-1}u\Big(|x_{N}^{k}|\big(x^{'}+
\frac{(x^{'})^{k}}{|x_{N}^{k}|}\big),|x_{N}^{k}|x_{N}\Big)
\end{eqnarray}
and
\begin{eqnarray}\notag
\Omega=\Big\{|x^{'}|<1,\frac{1}{2}<|x_{N}|<2\Big\}.
\end{eqnarray}

By the decay rate of $u$ in $x_{N}$ and the fact that $|x_{N}^{k}|\to\infty$, we have
\begin{eqnarray}\notag
|v^{k}(x)|\leq C|x_{N}^{k}|^{-1}|x_{N}|^{-\alpha}\leq C.
\end{eqnarray}
for any $x\in\Omega$ and for $k$ large enough. 
Since $f\in C^{2}$ and $f^{'}(0)=0$, we deduce that $|f(u)|/u^{2}$ is bounded in a neighbourhood of the origin, so
\begin{eqnarray}\notag
0\leq|\Delta v^{k}(x)|=|x_{N}^{k}|^{\alpha+1}\bigg|f\Big(u\Big(|x_{N}^{k}|\big(x^{'}+
\frac{(x^{'})^{k}}{|x_{N}^{k}|}\big),|x_{N}^{k}|x_{N}\Big)\Big)\bigg|\leq\\\notag
C|x_{N}^{k}|^{\alpha+1}u^{2}\Big(|x_{N}^{k}|\big(x^{'}+
\frac{(x^{'})^{k}}{|x_{N}^{k}|}\big),|x_{N}^{k}|x_{N}\Big)\leq
C|x_{N}^{k}|^{1-\alpha}|x_{N}|^{-2\alpha}\leq C
\end{eqnarray}
for any $x\in\Omega$ and for $k$ large enough .

By elliptic estimates we have that, for any ball $B\subset\subset\Omega$, for any $p>1$ and for any $k$,
\begin{eqnarray}\notag
||v^{k}||_{W^{2,p}(B)}\leq C(||v^{k}||_{L^{\infty}(\Omega)}+||\Delta v^{k}||_{L^{\infty}(\Omega)})\leq C.
\end{eqnarray}
Now we take $p>N$ and we conclude, by the Sobolev embedding $C^{1,\alpha}(B)\subset W^{2,p}(B)$ and since the ball is arbitrary, we have that $||\nabla v^{k}||_{L^{\infty}(\Omega)}$ is uniformly bounded with respect to $k$. 

On the other hand, an explicit computation gives that 
\begin{eqnarray}\notag
\Big|\nabla v^{k}\Big(0,\frac{x_{N}^{k}}{|x_{N}^{k}|}\Big)\Big|=|x_{N}^{k}|^{\alpha}
|\nabla u((x^{'})^{k},x_{N}^{k})|\geq k\to\infty,
\end{eqnarray}
a contradiction.

$(ii)$ The proof is the same as before, with the only difference that now we set
\begin{eqnarray}\notag
v^{k}(x^{'},x_{N})=|x_{N}^{k}|^{\alpha}u\Big(|x_{N}^{k}|\big(x^{'}+
\frac{(x^{'})^{k}}{|x_{N}^{k}|}\big),|x_{N}^{k}|x_{N}\Big).
\end{eqnarray}
The only point where we use that $\alpha\geq 2$ is to say that $||\Delta v^{k}||_{L^{\infty}(\Omega)}$ is uniformly bounded with respect to $k$. 
\end{proof}
By this lemma we see that, if $u$ fulfills (\ref{decay}) then the gradient satisfies
\begin{eqnarray}
|\nabla u(x)|\leq C|x_{N}|^{-(1+\sigma)} &\text{for $|x_{N}|\geq M$,}
\end{eqnarray}
for suitable constants $M>0$, $\sigma>0$, so it is possible to define the energy and the momentum, even if $f^{'}(0)=0$.

Now we recall that, under condition (\ref{decxn}), it is possible to start the moving plane procedure from the positive $x_{N}$ direction (see Proposition \ref{propstart}) and define $\ol$ as in (\ref{defol}). 

It is possible to show that, under assumption $(a)$ of Theorem \ref{thndim}, any profile is positive and we can find a profile $v$ that is symmetric with respect to $\ol$.

\begin{proposition}
If $u>\uol$ in $\sol$, then there exists a positive solution $v$ which is symmetric about the hyperplane $\{x_{N}=\ol\}$.
\end{proposition}

\begin{proof}
We take a sequence $x^{k}$ as in Proposition \ref{profile} and we define 
\begin{eqnarray}\notag
u^{k}(x)=u(x^{'}+(x^{'})^{k},x_{N}).
\end{eqnarray}
By the Ascoli-Arzel\'a theorem, up to a subsequence, $u^{k}$ converges to a non-negative solution $v$ to equation (\ref{eqp}). 

Now we want to prove that $v>0$. We point out that $|(x^{'})^{k}|\to\infty$. If not, by the boundedness of $x^{k}_{N}$, it would be possible to find a subsequence $x^{k}\to x^{\infty}$. Hence, passing to the limit in (\ref{relol}), we would get that $u(x^{\infty})\leq\uol(x^{\infty})$; since $u>\uol$ in $\sol$, we get that $x^{\infty}\in\partial{\sol}$, which contradicts the Hopf lemma.
In fact, still by $(\ref{relol})$ and by Lagrange theorem, we have that
\begin{eqnarray}\notag
0<\ulk(x^{k})-u(x^{k})=2(\lambda_{k}-x^{k})u^{k}_{N}(0,\xi^{k}),
\end{eqnarray}
for an appropriate $x^{k}_{N}<\xi^{k}<2\lambda_{k}-x^{k}_{N}$. Therefore, passing to the limit, we get that $u_{N}(0,\ol)\geq 0$, which contradicts the Hopf Lemma. 

We are now in position to show that $v>0$. In fact, $H(u^{k},x^{'})=H(u,x^{'}+(x^{'})^{k})\to H(v,x^{'})$, so $|H(v,x^{'})|>\gamma>0$, hence $v>0$.

It remains to prove that such a profile is symmetric. Since the translation is orthogonal to the $x_{N}$ direction, we have that $u^{k}\geq u^{k}_{\ol}$ in $\sol$, hence $v\geq v_{\lambda}$ in $\sol$. By the strong maximum principle, we can see that $v>v_{\lambda}$ or $v\equiv v_{\lambda}$ in $\sol$; we want to exclude the first possibility. In order to do so, we take a subsequence such that $x^{k}_{N}\to x^{\infty}_{N}$ and pass to the limit in (\ref{relol}), and we obtain that $v(0,x^{\infty})\leq v(0,x^{\infty})$. Now we observe that, if $v>v_{\ol}$, we get that $x^{\infty}_{N}=\ol$, which contradicts the Hopf Lemma, exactly as above. 

\end{proof}

In view of condition (\ref{decay}), the decay in $x_{N}$ holds both for $x_{N}\to\infty$ and for $x_{N}\to-\infty$, therefore we can also start the moving plane procedure from the left, and define
\begin{eqnarray}\notag
\ul=\sup\{\lambda_{0}: u-u_{\lambda}\geq 0\text{ in $\tilde{\Sigma}_{\lambda}$,} \forall\lambda\leq\lambda_{0}\},
\end{eqnarray}
where $\tilde{\Sigma}_{\lambda}=\{x_{N}>\lambda\}$.

As above, by construction, we get $\ul>-\infty$. Furthermore, we can prove that $\underline{\lambda}\leq\overline{\lambda}$. If not, we would have $u_{N}\geq 0$ in $\{x_{N}<\underline\lambda\}$ and $u_{N}\leq 0$ in $\{x_{N}>\overline\lambda\}$, so $u_{N}=0$ in
$\{\overline\lambda<x_{N}<\underline\lambda\}$. By the strong maximum principle we get, for instance, that $u_{N}\equiv 0$ in $\Sigma_{\underline{\lambda}}$, hence $u\equiv 0$.
\\

\noindent\textbf{Remark 9.} If $\ul=\ol$, then $u$ is symmetric with respect to $x_{N}$, that is $u=u(x^{'},|x_{N}-\ol|)$.
\\

To conclude the proof of Theorem \ref{thndim}, we have to rule out the possibility $\ul<\ol$. In order to do so, we prove the following proposition.



\begin{proposition}
Let $u>0$ be a bounded positive solution to equation (\ref{eqp}) satisfying (\ref{decay}), with $f$ as in (\ref{condf}). Assume furthermore that

$(i)$ $|H(u,x^{'})|\geq\gamma>0$ for $|x|^{'}$ large enough

$(ii)$ there exists $\mu$ such that $E_{\mu}(u,x^{'})\to 0$ for $x^{'}\to\infty$.

Then $u(x)=u(x^{'},|x_{N}-\lambda|)$, for a suitable $\lambda\in\mathbb{R}$ (that is, $u$ is symmetric in $x_{N}$).
\label{propH}
\end{proposition}

\begin{proof}
We divide the proof in two steps.

$(i)$ If $u>u_{\ol}$ in $\sol$, then $\ol=\mu$.
  
We define
\begin{eqnarray}\notag
\tilde{u}(x)=u(x^{'},x_{N}+\ol-\mu)
\end{eqnarray}
and
\begin{eqnarray}\notag
\tilde{u}^{k}(x)=\tilde{u}(x^{'}+(x^{'})^{k},x_{N}).
\end{eqnarray}
It is worth to remark that the profile of the translated solution $\tilde{u}$ coincides with the translation of the profile $\tilde{v}$, that is $\tilde{u}_{k}\to\tilde{v}$, up to a subsequence. Since $v$ is symmetric about the hyperplane $\{x_{N}=\ol\}$, $\tilde{v}$ is symmetric about the hyperplane $\{x_{N}=\mu\}$, therefore, if we set
\begin{equation}\notag
g(x)=\frac{1}{2}\big(u_{N}^{2}-
|\nabla_{x^{'}}u|^{2}\big)-F(u),
\end{equation}
then we have
\begin{eqnarray}\notag
0=E_{\mu}(\tilde{v},x^{'})=\lim_{k\to\infty}E_{\mu}(\tilde{u}^{k},x^{'})=
\lim_{k\to\infty}E_{\mu}(\tilde{u},x^{'}+(x^{'})^{k})=\\\notag
\lim_{k\to\infty}\int_{-\infty}^{\infty}(x_{N}-\mu)g(x^{'}+(x^{'})^{k},x^{N}+
\ol-\mu)dx_{N}=\\\notag
\lim_{k\to\infty}\bigg\{\int_{-\infty}^{\infty}(z_{N}-\mu)g(x^{'}+(x^{'})^{k},
z_{N})dz_{N}-\int_{-\infty}^{\infty}(\ol-\mu)g(x^{'}+(x^{'})^{k},
z_{N})dz_{N}\bigg\}=\\\notag
\lim_{k\to\infty}\Big\{E_{\mu}(u,x^{'}+(x^{'})^{k})-(\ol-\mu)H(u,x^{'}+
(x^{'})^{k})\Big\}=-(\ol-\mu)H(v,x^{'}).
\end{eqnarray}
Since $H(v,x^{'})\neq 0$, we have $\ol=\mu$.



$(ii)$ $\ul=\mu=\ol$.

In order to prove the statement, we start the reflection from the left and obtain that either $u$ is symmetric about the hyperplane $\{x_{N}=\ul\}$ or $u>u_{\ul}$ in $\stul$; in the second case, exactly as in Proposition \ref{profile}, we are able to construct a sequence $\ul<\lambda_{k}<\ul+1/k$ and a sequence of points $s^{k}\in\stlk$ such that $u(s^{k})<u_{\lambda_{k}}(s_{k})$, with $|(s^{'})^{k}|\to\infty$ and $\{s^{k}_{N}\}$ bounded. Passing to the limit, we get a profile $w$ which is symmetric about the hyperplane $\{x_{N}=\ul\}$. Since $\ol=\mu$, we have
\begin{eqnarray}\notag
0=\lim_{k\to\infty}E_{\ol}(u,x^{'}+(s^{'})^{k})=\lim_{k\to\infty}\int_{-\infty}
^{\infty}(x_{N}-\ol)g(x^{'}+(s^{'})^{k},x_{N})dx_{N}=\\\notag
\lim_{k\to\infty}\bigg\{\int_{-\infty}^{\infty}(x_{N}-\ul)g(x^{'}+(s^{'})^{k},
x_{N})dx_{N}-\int_{-\infty}^{\infty}(\ol-\ul)g(x^{'}+(s^{'})^{k},
x_{N})dx_{N}\bigg\}=\\\notag
E_{\ul}(w,x^{'})-(\ul-\ol)H(w,x^{'})=-(\ul-\ol)H(w,x^{'}).
\end{eqnarray}
Since $H(w,x^{'})\neq 0$, then $\ul=\ol$.

\end{proof}

Now we can recollect our results to conclude the proof of Theorem \ref{thndim}.

\begin{proof}
The idea is to apply Proposition \ref{propH}. Therefore, we have to check that hypothesis $(i)$ and $(ii)$ are satisfied.
As first, we will prove that $H(u,x^{'})$ tends to a finite positive limit as 
$|x^{'}|\to\infty$. In order to do so, we take an arbitrary sequence $|(x^{'})^{k}|\to\infty$ and we prove that, up to a subsequence, $H(u,(x^{'})^{k})$ converges to a positive limit which is indipendent of the chosen sequence.

By the Arzel\'a-Ascoli theorem, for any sequence $|(x^{'})^{k}|\to\infty$, we can find a subsequence such that $u^{k}(x)=u(x^{'}+(x^{'})^{k},x_{N})$ converges to a nonnegative profile $v$, which still verifies $-\Delta v=f(v)$. By hypothesis (a), we have that $v>0$;
by Lemma \ref{lemmaprof1d}, we get that $v$ is one-dimensional, that is $v=v(x_{N})$. Moreover, by condition (\ref{decay}) we get that $v(x_{N})\leq C|x_{N}|^{-(1+\sigma)}$ for $|x_{N}|\geq M$.

As a consequence, $v$ is a solution to problem (\ref{ODE}) for which the energy $H(v)$ and the momentum $E_{\lambda}(v)$ are well defined and finite.
Moreover,
\begin{eqnarray}\notag
H(u,(x^{'})^{k})=H(u^{k},0)\to H(v,0)=H(v)=\int_{-\infty}^{\infty}(v^{'})^{2}>0.
\end{eqnarray}
By the uniqueness of the positive solution to (\ref{ODE}), proven in Proposition \ref{propODE}, we get that the limit does not depend on the particular choice of the sequence $|(x^{'})^{k}|\to\infty$, hence
\begin{eqnarray}\notag
H(u,x^{'})\to H(v)>0 &\text{as $|(x^{'})|\to\infty$.}
\end{eqnarray}
In the same way as above, it is possible to prove that $E_{0}(u,x^{'})\to E_{0}(v)$ as $|(x^{'})|\to\infty$. Therefore
\begin{eqnarray}\notag
E_{\mu}(u,x^{'})=E_{0}(u,x^{'})-\mu H(u,x^{'})\to E_{0}(v)-\mu H(v),
\end{eqnarray}
so it is enough to take $\mu=E_{0}(v)/H(v)$. This concludes the proof of Theorem \ref{thndim}.
\end{proof}

Now we prove Theorem \ref{th2dim}. In the proof, we will use a result by Malchiodi, Gui and Xu (see \cite{GMX}, Proposition $2$).
If $N=2$, they show that $H(u,x_{1})$ is actually independent of $x_{1}$, hence it may be referred to as $H(u)$. If $H(u)\neq 0$, we can apply Proposition \ref{propH} with $\mu=E_{0}(u)/H(u)$, and the proof is finished.

It remains to deal with the case $H(u)=0$. We claim that in this case $u$ is radially symmetric, that is, up to a translation, $u=u(|x|)$, where $x=(x_{1},x_{2})\in\R^{2}$.

\begin{proposition}
In the hypothesis of Theorem \ref{th2dim}, if $H(u)=0$, then $u$ is radially symmetric.
\label{propH0}
\end{proposition}

\begin{proof}
In view of Lemma \ref{lemmasolrad}, it is enough to show that any profile is identically $0$ and apply the result by Gidas, Ni and Nirenberg in \cite{GNN}.

Assume, by contradiction, that one can find a sequence $|x_{1}^{k}|\to\infty$ whose correspondent profile $v$ is stricly positive. By Lemma \ref{lemmaprof1d}, this profile is one-dimensional, therefore it is the unique (up to a translation) solution to Problem (\ref{ODE}), hence we already know that $H(v)=\int_{-\infty}^{\infty}(v^{'})^{2}>0$. On the other hand, by the dominated convergence theorem, we have that $H(v)=H(u)=0$, a contradiction.
\end{proof}

\section{Solutions decaying in $N-1$ variables}
Now we are considering solutions to equation (\ref{eqp}) fulfilling (\ref{limx'}). The nonlinearity will always satisfy (\ref{condf}), sometimes it will be required to be of class $C^{2}$, sometimes $C^{1}$ will be enough. 

For such solutions, we define the energy-like functional
\begin{eqnarray}\notag
\mathcal{H}(u,x_{N})=\int_{\R^{N-1}}\frac{1}{2}\big(|\nabla_{x^{'}}u|^{2}-u_{N}^{2}\big)-F(u)
dx^{'}.
\end{eqnarray}

We point out that, in order for such a functional to be well defined and finite, we need some further information about the decay rate of $u$, for example it is enough to consider solutions $u$ fulfilling (\ref{decN}).
\\

\noindent\textbf{Remark 10.} If $f^{'}(0)<0$, any solution satisfying (\ref{limx'}) actually decays exponentially in $x^{'}$, and the same is true for the gradient, that is
\begin{eqnarray}\notag
u(x),|\nabla u(x)|\leq Ce^{-\gamma|x^{'}|} &\text{for $|x^{'}|\geq M$},
\label{expdecayx'}
\end{eqnarray}
for some $M>0$, $\gamma>0$, and this is true in any dimension $N\geq 2$, hence there are no problems to define $\mathcal{H}(u,x_{N})$.
\\

It is interesting to understand what happens in the case $f^{'}(0)=0$. It turns out that, at least in dimension $N\geq 5$, if $f\in C^{2}$, any solution fulfilling (\ref{limx'}) actually decays fast enough in $x^{'}$, so it is still possible to define $\mathcal{H}(u,x_{N})$. In dimension $2\leq N\leq 4$, it is possible to do the same under hypothesis (\ref{decN}).

Moreover, we recall that in \cite{GMX} Malchiodi, Gui and Xu showed that $\mathcal{H}(u,x_{N})$ actually depends only on $u$, hence it will be referred to simply as $\mathcal{H}(u)$.


\begin{lemma}
Let $u>0$ be a bounded solution to (\ref{eqp}), with $f\in\ C^{1}$ satisfying (\ref{condf}) and $f^{'}(0)=0$. Assume furthermore that (\ref{limx'}) holds. Then
\begin{eqnarray}\notag
\nabla u(x^{'},x_{N})\to 0 &\text{as $|x^{'}|\to\infty$, uniformly in $x_{N}$}.
\end{eqnarray}

\label{lemmaest}

\end{lemma}


\begin{proof} Assume, by contradiction, that it is possible to find a $\delta>0$ and a sequence $|(x^{'})^{k}|\to\infty$ such that 
\begin{eqnarray}\notag
\sup_{x_{N}}|\nabla u((x^{'})^{k},x_{N})|\geq 2\delta.
\end{eqnarray}
So we can take a sequence $x_{N}^{k}\in\R$ such that $|\nabla u((x^{'})^{k},x_{N}^{k})|\geq\delta$ and define $u^{k}(x)=u(x+x^{k})$. Up to a subsequence, $u^{k}\to v$ in $C^{2,\alpha}_{loc}(\R^{N})$, and $v\geq 0$ is still a solution to equation (\ref{eqp}). Now we observe that, on the one hand
\begin{eqnarray}\notag
u^{k}(0)=u(x^{k})\to 0=v(0),
\end{eqnarray}
hence, by the strong maximum principle, $v\equiv 0$. On the other hand,
\begin{eqnarray}\notag
\delta\leq|\nabla u((x^{'})^{k},x_{N}^{k})|=|\nabla u^{k}(0)|\to|\nabla v(0)|=0,
\end{eqnarray}
a contradiction.
\end{proof}

\begin{lemma}
Let us denote $x=(y,z)\in\R^{M}\times\R^{N-M}$. Assume that $N-M\geq 3$.
Let $u>0$ be a bounded $C^{2}(\R^{N})$ function such that $-\Delta u\leq 0$ for $|z|\geq r$, for some $r>0$. Assume furthermore that 
\begin{eqnarray}
u(y,z)\to 0 &\text{as $|z|\to\infty$, uniformly in $y$}
\label{decz}
\end{eqnarray}
Then
\begin{eqnarray}
u(x)\leq C|z|^{2-(N-M)} &\text{for $|z|\geq R$}
\label{condM-N}
\end{eqnarray}
for a suitable constant $R>0$.
\label{lemmadecarm}
\end{lemma}

\begin{proof}
We will give an estimate of $u$ by dominating it with a barrier. In this construction we use the function 
$v(y,z)=|z|^{2-(N-M)}$, because we know that $v>0$ and $\Delta v=0$ on $\R^{N}$, for $N-M\geq 3$.
We observe that, for any $\sigma>0$ and $\lambda\in\R$,
\begin{eqnarray}\notag
-\Delta\big(u-(\sigma+\lambda v)\big)\leq 0 &\text{for $|z|\geq r$}
\end{eqnarray}
By the decay in $|z|$, we deduce that, for any $\varepsilon>0$, we can find $\rho=\rho(\varepsilon)>0$ such that $u(y,z)<\varepsilon$ for $|z|\geq\rho$. Now we set $R=\max\{\rho,r\}$. We fix $0<\sigma<\varepsilon$, $x_{0}=(y_{0},z_{0})$ such that $|z_{0}|>R$ and we take $A>|z_{0}|$ so large that $u<\sigma$ for $|z|\geq A$. Hence we have  
\begin{eqnarray}\notag
\begin{cases}
u<\sigma<\sigma+\lambda R^{2-(N-M)} &\text{for $|z|=A$}\\\notag
u<\varepsilon<\lambda R^{2-(N-M)}<\sigma+\lambda R^{2-(N-M)} &\text{for $|z|=R$}
\end{cases}
\end{eqnarray}
if we choose $\lambda>\varepsilon R^{N-M-2}$. Therefore, by the maximum principle for possibly unbounded domains (see \cite{BCN}, Lemma $2,1$) applied to the region 
$C=\{x\in\R^{N}:R<|z|<A\}$, we get $u\leq\sigma+\lambda v$ on $C$, in particular $u(x_{0})\leq\sigma+\lambda v(x_{0})$. Letting $\sigma\to 0$, we have the statement.
\end{proof}

\begin{corollary}
Let $u>0$ be a bounded solution to (\ref{eqp}), with $f$ satisfying (\ref{condf}). 

$(i)$ If $N\geq 4$ and $u$ satisfies (\ref{limx'}), then 
\begin{eqnarray}
u(x^{'},x_{N})\leq C|x^{'}|^{3-N} &\text{for $|x^{'}|\geq M$}
\label{cond3-N}
\end{eqnarray} 
for a suitable constant $M>0$.

$(ii)$ If $N\geq 3$ and $u(x)\to 0$ as $|x|\to\infty$, then
\begin{eqnarray}
u(x)\leq C|x|^{2-N} &\text{for $|x|\geq M$}
\label{decarm}
\end{eqnarray}
\label{corarm}
\end{corollary}
\begin{proof}
It is enough to apply Lemma \ref{lemmadecarm} with $M=1$ in case $(i)$ and with $M=0$ in case $(ii)$.
\end{proof}

\begin{lemma}
Let $u>0$ be a bounded solution to (\ref{eqp}), with $f\in C^{2}(\R)$ satisfying (\ref{condf}). 

$(i)$ If $N\geq 5$ and $u$ satisfies (\ref{limx'}), then
\begin{eqnarray}
|\nabla u(x^{'},x_{N})|\leq C|x^{'}|^{2-N} &\text{for $|x^{'}|\geq M$}
\label{cond2-N}
\end{eqnarray}
for a suitable constant $M>0$.

$(ii)$ If $N\geq 4$ and $u(x)\to 0$ as $|x|\to\infty$, then
\begin{eqnarray}
|\nabla u(x)|\leq C|x|^{1-N} &\text{for $|x|\geq M$}
\label{decarmgrad}
\end{eqnarray}
for a suitable constant $M>0$.
\label{lemmagradarm}
\end{lemma}
\begin{proof}
It is enough to apply statement $(ii)$ of Lemma \ref{lemmaestgrad} with $\alpha=N-3$ in case $(i)$ and $\alpha=N-2$ in case $(ii)$.
\end{proof}

\begin{lemma}
Let $u>0$ be a bounded solution to the (\ref{eqp}), with $f\in C^{2}(\R)$ satisfying the (\ref{condf}). Assume furthermore that (\ref{limx'}) holds.

$(i)$ Let $N\geq 5$. Then $\mathcal{H}(u)$ is well defined and finite.

$(ii)$ If $2\leq N\leq 4$, the same is true under condition (\ref{decN}).
\label{lemmadefH}
\end{lemma}

\begin{proof}
As above, we can assume that $f^{'}(0)=0$, otherwise the result follows from the exponential decay.

$(i)$ Applying Lemma \ref{lemmagradarm}, we get that
\begin{eqnarray}\notag
\int_{|x^{'}|\geq M}u_{j}^{2}dx^{'}\leq C\int_{M}^{\infty}r^{2(2-N)}r^{N-2}dr
\end{eqnarray}
that is finite because $N\geq 5$. 

By the assumption $f^{'}(0)=0$ and $f\in C^{2}$, we get that $F(u)/u^{3}$ is bounded in a neighbourhood of the origin. If $N\geq 5$, this yields that
\begin{eqnarray}\notag
\Big|\int_{|x^{'}|\geq R}F(u)dx^{'}\Big|\leq C\int_{R}^{\infty}r^{3(3-N)}r^{N-2}dr<\infty
\end{eqnarray}

$(ii)$ If $2\leq N\leq 4$, condition (\ref{decN}) yields that
\begin{eqnarray}\notag
\int_{|x^{'}|\geq M}u_{j}^{2}dx^{'}\leq C\int_{M}^{\infty}r^{-(N-1+\sigma)}r^{N-2}dr<\infty
\end{eqnarray}
and
\begin{eqnarray}\notag
\Big|\int_{|x^{'}|\geq R}F(u)dx^{'}\Big|\leq C\int_{R}^{\infty}r^{-3\frac{N-1+\sigma}{2}}r^{N-2}dr<\infty
\end{eqnarray}

\end{proof}

Let $N\geq 4$. For a solution $u>0$ to (\ref{eqp}) such that $u(x)\to 0$ as $|x|\to\infty$, we define
\begin{eqnarray}\notag
J(u)=\int_{\R^{N}}\frac{1}{2}|\nabla u|^{2}-F(u)dx.
\end{eqnarray}


We point out that, if $f^{'}(0)<0$, any positive solution decaying to $0$ decays exponentially, so  the restirction on the dimension is not necessary, we can define $J(u)$ for any $N\geq 1$.

Anyway by Corollary \ref{corarm} and Lemma \ref{lemmagradarm}, in dimension $N\geq 4$, even if $f^{'}(0)=0$, the fact that $u\to 0$ as $|x|\to\infty$ is sufficient to guarantee that $J(u)$ is well defined and finite. In fact
\begin{eqnarray}\notag
\int_{\R^{N}}|\nabla u|^{2}\leq C\int_{0}^{\infty}r^{2(1-N)}r^{N-1}dr<\infty
\label{Jgrad}
\end{eqnarray} 
and
\begin{eqnarray}\notag
\bigg|\int_{\R^{N}}F(u)\bigg|\leq C\int_{0}^{\infty}r^{3(2-N)}r^{N-1}dr<\infty.
\label{JFu}
\end{eqnarray}

In dimension $1\leq N\leq 3$, the decay to $0$ is not sufficient to define $J(u)$, at least if $f^{'}(0)=0$. In order to do so, we have to assume some further conditions about the decay of $u$, for instance
\begin{eqnarray}
u(x),|\nabla u(x)|\leq C|x|^{-\frac{N+\sigma}{2}} &\text{for $|x|\geq M$}
\label{condH1}
\end{eqnarray}
for appropriate constants $M>0$, $\sigma>0$.

In next lemma, we will compute explicitly $J(u)$, and we will see that $J(u)>0$.

\begin{lemma}
Let $u>0$ be a solution to the problem
\begin{eqnarray}
\begin{cases}
-\Delta u=f(u) &\text{in $\R^{N}$}\\\notag
u>0\\\notag
u(x)\to 0 &\text{as $|x|\to\infty$}
\end{cases}
\label{probrad}
\end{eqnarray}
with $f\in C^{2}(\R)$ satisfying (\ref{condf}).

$(i)$ If $N\geq 4$, then 
\begin{eqnarray}
J(u)=\frac{1}{N}\int_{\R^{N}}|\nabla u|^{2}>0
\label{explicitJ}
\end{eqnarray}

$(ii)$ If $1\leq N\leq 3$, the same formula holds if $u$ fulfills condition (\ref{condH1}) and $f\in C^{1}$.
\label{lemmadec}
\end{lemma}

\noindent\textbf{Remark 11.} If $f\in C^{1}(\R)$ with $f^{'}(0)<0$, thanks to the exponential decay (\ref{expdecayx'}), formula (\ref{explicitJ}) holds true in any dimension $N\geq 1$. 

\begin{proof}
If $N=1$, condition (\ref{condH1}) guarantees that $J(u)$ is well defined and finite. By statement $(ii)$ of Lemma \ref{lemmaODE}, $\frac{1}{2}(u^{'})^{2}+F(u)=0$,
therefore $J(u)=\int_{-\infty}^{\infty}(u^{'})^{2}>0$, unless $u\equiv 0$.

Now we observe that, in any dimension $N\geq 2$ and for any nonlinearity $f$ fulfilling the (\ref{condf}), any solution to (\ref{probrad}) is radially symmetric, that is, up to a translation, $u(x)=v(|x|)$, where $v$ satisfies that ODE
\begin{eqnarray}\notag
-v^{''}-\frac{N-1}{r}v^{'}=f(v)
\end{eqnarray} 
We multilpy the ODE by $v^{'}r^{N}$ and integrate to obtain
\begin{eqnarray}\notag
-\int_{0}^{\infty}v^{''}v^{'}r^{N}dr-(N-1)\int_{0}^{\infty}(v^{'})^{2}r^{N-1}dr=
\int_{0}^{\infty}f(v)v^{'}r^{N}dr
\end{eqnarray}
Integrating by parts we get
\begin{eqnarray}\notag
\int_{0}^{\infty}f(v)v^{'}r^{N}dr=\big[F(v)r^{N}\big]_{0}^{\infty}-
N\int_{0}^{\infty}F(v)r^{N-1}dr
\end{eqnarray}
and
\begin{eqnarray}\notag
2\int_{0}^{\infty}v^{''}v^{'}r^{N}dr=\big[(v^{'})^{2}r^{N}\big]_{0}^{\infty}-
N\int_{0}^{\infty}(v^{'})^{2}r^{N-1}dr
\end{eqnarray}
If $f^{'}(0)<0$, thanks to the exponential decay, all integrals are well defined and finite and all boundary terms vanish. Finally, we get
\begin{eqnarray}
\frac{N-2}{2N}\int_{0}^{\infty}(v^{'})^{2}r^{N-1}dr=
\int_{0}^{\infty}F(v)r^{N-1}dr
\end{eqnarray}
If $N=2$, we already see that $\int_{0}^{\infty}F(v)r^{N-1}dr=0$, hence $J(u)=\frac{1}{2}\int_{0}^{\infty}(v^{'})^{2}rdr>0$. In higher dimension, a computation show that $J(u)=\frac{1}{N}\int_{0}^{\infty}(v^{'})^{2}r^{N-1}dr>0$.

If $f^{'}(0)=0$, we have no exponential decay, so it is harder to verify that all the integrals are well defined and finite and that the boundary terms vanish. In order to do so, in dimension $1\leq N\leq 3$ we use condition (\ref{condH1}), while in higher dimension, by Corollary \ref{corarm} and \ref{lemmagradarm}, the decay at infinity is enough to guarantee (\ref{decarm}) and (\ref{decarmgrad}), hence all the integrals are well defined and finite and the boundary terms vanish.

\end{proof}

Now we prove Theorem \ref{thdec}.

\begin{proof}
As first, we point out that, in dimension $N\geq 5$, by Lemma \ref{lemmadefH}, condition (\ref{limx'}) is enough to guarantee suitable decay to define $H(u)$.
For any sequence $x_{N}^{k}\to\infty$, it is possible to find a subsequence such that $u^{k}(x)=u(x^{'},x_{N}+x_{N}^{k})$ converges to a profile $\to u^{\infty}$ in the $C^{2,\alpha}_{loc}$ sense. By hypothesis (\ref{decn}), 
\begin{eqnarray}\notag
u^{\infty}(x^{'}_{0},0)=\lim_{k\to\infty}u^{k}(x^{'}_{0},0)=
\lim_{k\to\infty}u(x^{'}_{0},x_{N}^{k})=0
\end{eqnarray}
hence $u^{\infty}\equiv 0$. Since the sequence is arbitrary, by Lemma \ref{lemmasolrad}, $u(x^{'},x_{N})\to 0$ as $x_{N}\to\infty$, uniformly in $x^{'}$, so we can apply Proposition \ref{propstart} to begin the moving plane procedure (see Remark $4$). Now, since we do not know the behaviour of $u$ for $x_{N}\to-\infty$, we have to be careful to exclude the case $\ol=-\infty$. Assume, by contradiction, that $\ol=-\infty$. Then we get $u_{N}\leq 0$ and therefore, since $u_{N}$ satisfies $-\Delta u_{N}=f^{'}(u)u_{N}$, by the strong maximum principle we have $u_{N}<0$, hence it is possible to define, for any $x^{'}\in\R^{N-1}$,
\begin{eqnarray}\notag
\underline{u}(x^{'})=\lim_{x_{N}\to-\infty}u(x^{'},x_{N}).
\end{eqnarray} 
By the Arzelï¿½-Ascoli theorem, it is possible to check that the convergence holds in $C^{2}_{loc}$, hence the profile $\underline{u}$ satisfies
\begin{eqnarray}\notag
\begin{cases}
-\Delta\underline{u}=f(\underline{u}) &\text{in $\R^{N-1}$}\\\notag
\underline{u}>0\\\notag
\underline{u}(x^{'})\to 0 &\text{as $|x^{'}|\to\infty$}\notag
\end{cases}
\end{eqnarray}
Therefore, applying Lemma \ref{lemmadec} to $\underline{u}$, we get that $J(\underline{u})>0$.

However, by relation (\ref{Jgrad}), $J(\underline{u})$ is well defined and finite and
\begin{eqnarray}\notag
J(\underline{u})=\int_{\R^{N-1}}\frac{1}{2}|\nabla \underline{u}|^{2}-F(\underline{u})dx^{'}=\\\notag
\lim_{x_{N}}\mathcal{H}(u,x_{N})=\mathcal{H}(u)=\lim_{x_{N}\to\infty}
\mathcal{H}(u,x_{N})=0,
\end{eqnarray}
a contradiction.

As a consequence, we get that $\ol\in\R$ and $u-u_{\ol}\geq 0$ in $\Sigma_{\ol}$. By the strong maximum principle, we have that $u>u_{\ol}$ or $u\equiv u_{\ol}$ in $\Sigma_{\ol}$. To conclude the proof of the theorem we have to exclude the first possibility.

Assume, by contradiction, that $u>u_{\ol}$ in $\Sigma_{\ol}$. By Proposition \ref{profile1}, applied to the case $M=0$, we can find a sequence of real numbers $\ol-1/k\leq\lambda_{k}<\ol$ and a bounded sequence of points $x^{k}\in\slk$, such that
\begin{eqnarray}\notag
 u(x^{k})<u_{\lambda_{k}}(x^{k}).
\end{eqnarray} 
Up to a subsequence, $x^{k}\to x^{\infty}$, therefore $u(x^{\infty})\leq u_{\ol}(x^{\infty})$. Since we are assuming that $u>u_{\ol}$ in $\Sigma_{\ol}$, we get that $x^{\infty}_{N}=\ol$, but this is a contradiction to the Hopf lemma, as above. To conclude, we observe that the symmetry in the $x_{N}$ variable yields that
\begin{eqnarray}\notag
u(x^{'},x_{N})\to 0 &\text{as $x_{N}\to-\infty$, uniformly in $x^{'}$},
\end{eqnarray}
hence, by the result by Gidas, Ni and Nirenberg in \cite{GNN}, we get the radial symmetry.
\end{proof}





\end{document}